\documentclass[11pt]{amsart}
\usepackage{mathrsfs}
\usepackage{amsfonts}
\usepackage{color}
\usepackage{latexsym,amsmath,amssymb}

 \renewcommand{\epsilon}{\varepsilon}

\newtheorem{theorem}{Theorem}[section]
 
 \newtheorem{lemma}[theorem]{Lemma}

 \newtheorem{proposition}[theorem]{proposition}
 \newtheorem{Proposition}[theorem]{Proposition}
\newtheorem{deff}[theorem]{Definition}
 \newtheorem{rem}[theorem]{Remark}
  \newtheorem{claim}[theorem]{Claim}
 \newcommand{\bth}{\begin{theorem}}
 \newcommand{\ble}{\begin{lemma}}
 \newcommand{\bcor}{\begin{corr}}
 \newcommand{\bdeff}{\begin{deff}}
 \newcommand{\bprop}{\begin{proposition}}
 \newcommand{\ele}{\end{lemma}}
 \newcommand{\ecor}{\end{corr}}
 \newcommand{\edeff}{\end{deff}}
 
 \newcommand{\eprop}{\end{proposition}}

 \renewcommand{\Pi}{\varPi}

 \renewcommand{\epsilon}{\varepsilon}

\numberwithin{equation}{section}

\newtheorem{lem}{Lemma}[section]

\newtheorem{definition}[lem]{Definition}

\DeclareMathOperator{\diag}{diag}
\DeclareMathOperator{\diver}{div}
\title
[Nonlinear second boundary problem]{On the entire self-shrinking solutions to  Lagrangian mean curvature flow II}

\thanks{The first author is supported by National Natural Science Foundation of China (No. 11771103) and Guangxi Natural Science Foundation (2017GXNSFFA198017). The second author is supported by National Natural Science Foundation of China (No. 11861016). The third author is supported by National Postdoctoral Program for Innovative Talents of China (No. BX201700007) and National Natural Science Foundation of China (No. 11701326, 11701326).}

\date{}

\begin{document}
\maketitle

\begin{center}Rongli Huang \footnote{ School of Mathematics and Statistics, Guangxi Normal University, Guangxi 541004, China. ronglihuangmath@gxnu.edu.cn}
$\cdot$ Qianzhong Ou \footnote{The corresponding author. School of Mathematics and Statistics, Guangxi Normal University, Guangxi 541004, China. ouqzh@gxnu.edu.cn}
$\cdot$ Wenlong Wang \footnote{Key Laboratory of Pure and Applied Mathematics, School of Mathematical Sciences, Peking University, Beijing 100871, China. wangwl@pku.edu.cn}
\end{center}


\begin{abstract}
We show Bernstein type results for the entire self-shrinking solutions to Lagrangian mean curvature flow
in  $(\mathbb{R}^n\times\mathbb{R}^n, g_\tau)$.  The proofs rely on a priori estimates and
barriers construction.
\end{abstract}





\section{Introduction}
Let
$$g_\tau=\sin \tau\,\delta_0+\cos \tau\,g_0,\ \ \tau\in\left[0,\frac{\pi}{2}\right]$$
be the linearly combined metric of the standard Euclidean metric
$$\delta_0=\sum_{i=1}^n dx_i\otimes dx_i+\sum_{j=1}^ndy_j\otimes dy_j$$
and the pseudo-Euclidean metric
$$g_{0}=\sum_{i=1}^n dx_i\otimes dy_i+\sum_{j=1}^n dy_j\otimes dx_j$$
on $\mathbb{R}^{n}\times \mathbb{R}^{n}$. This linearly combined metric has been introduced by M. Warren in \cite{WM}
to study a family of special Lagrangian equations.

Consider the following fully nonlinear parabolic equations:
\begin{equation}\label{e1.1}
v_{t}=F_\tau(\lambda(D^2 v)),\quad x\in\mathbb{R}^{n},
\end{equation}
where
\begin{equation*}
F_{\tau}(\lambda):=\left\{ \begin{aligned}
&\frac{1}{2}\sum_{i=1}^n\ln\lambda_{i}, &&\tau=0, \\
& \frac{\sqrt{a^2+1}}{2b}\sum_{i=1}^n\ln\frac{\lambda_{i}+a-b}{\lambda_{i}+a+b},  &&0<\tau<\frac{\pi}{4},\\
& -\sqrt{2}\sum_{i=1}^n\frac{1}{1+\lambda_{i}}, &&\tau=\frac{\pi}{4},\\
& \frac{\sqrt{a^2+1}}{b}\sum_{i=1}^n\arctan\frac{\lambda_{i}+a-b}{\lambda_{i}+a+b},  \ \ &&\frac{\pi}{4}<\tau<\frac{\pi}{2},\\
& \sum_{i=1}^n\arctan\lambda_{i}, &&\tau=\frac{\pi}{2},
\end{aligned} \right.
\end{equation*}
$a=\cot \tau$, $b=\sqrt{|\cot^2\tau-1|}$,  $x=(x_{1},x_{2},\cdots,x_{n})$, $v=v(x,t)$ and $\lambda(D^2 v)=(\lambda_1,\cdots, \lambda_n)$ are the eigenvalues of the Hessian matrix $D^2 v$.

For an admissible solution $v$ to \eqref{e1.1} (see Definition \ref{d1.10} below), let $M_t=\{(x,Dv(x,t))\,|\,x\in\mathbb{R}^n\}$ be the gradient graph of $v(\cdot,t)$ and $X_t$ be the embedding map of $M_t$.  Since $v$ is admissible, $M_t$ is spacelike in $(\mathbb{R}^{n}\times\mathbb{R}^{n},g_\tau)$ (see \eqref{induced metric} below). And it is not hard to see that $M_t$ is Lagrangian with respect to the usual symplectic structure of $\mathbb{R}^{2n}$.
By Proposition 2.1, there exists a family of diffeomorphisms
\begin{equation*}
\psi_{t}: \mathbb{R}^{n}\rightarrow\mathbb{R}^{n}
\end{equation*}
such that the family of embeddings
\begin{equation*}
\tilde X_t\triangleq X_t\circ\psi_t=(\psi_{t}, Dv(\psi_{t},t))
\end{equation*}
is a solution to the Lagrangian mean curvature flow in $(\mathbb{R}^n\times\mathbb{R}^n, g_\tau)$:
\begin{equation}\label{e1.2}
\frac{d\tilde X_t}{dt}=H,
\end{equation}
where $H$ is the mean curvature vector of $M_{t}$ at $\tilde X_t$ in $(\mathbb{R}^n\times\mathbb{R}^n, g_\tau)$.

An important class of solutions to the mean curvature flow are self-shrinking solutions, whose profiles, self-shrinkers, satisfy
\begin{equation}\label{e1.3}
H=-\frac{1}{2}{X^\perp},
\end{equation}
where $^\perp$ stands for the orthogonal projection into the normal bundle (see \cite{Huisken}).
The main goal of this paper is to give fine classifications of Lagrangian self-shrinkers with entire potentials in $(\mathbb{R}^n\times\mathbb{R}^n, g_\tau)$.

For such a Lagrangian  self-shrinker $M=\{(x,Du(x))\,|\,x\in\mathbb{R}^n\}$, by Propositions 2.2--2.5,  up to an additive constant, (\ref{e1.3}) isequivalent to the following equation:
\begin{equation}\label{e1.4}
F_\tau(\lambda(D^2 u))=-u+\frac{1}{2}\langle x, Du\rangle,\quad x\in\mathbb{R}^{n},
\end{equation}
where  $\langle\cdot,\cdot\rangle$ denotes the standard inner product on $\mathbb{R}^{n}$.

The Bernstein type results of (\ref{e1.4}) have been obtained by several authors for  $\tau=0$ and  $\tau=\frac{\pi}{2}$ as following.

For $\tau=0$, (\ref{e1.4}) is the  Monge-Amp\`{e}re type equation
\begin{equation}\label{e1.5}
\frac{1}{2}\ln\det D^2u=-u+\frac{1}{2}\langle x, Du\rangle,\quad x\in\mathbb{R}^{n}.
\end{equation}
The first author \cite{HR1} showed that any entire solution to (\ref{e1.5}) leads to  a self-shrinking
solution to the Lagrangian mean curvature flow in the pseudo-Euclidean space  $(\mathbb{R}^n\times\mathbb{R}^n, g_0)$.
Several authors  proved that an entire smooth strictly convex solution to (\ref{e1.5}) must be a quadratic polynomial under the decay condition on the Hessian of $u$ at infinity by different methods.
Huang-Wang \cite{HW} used the flow method by carrying out Calabi's third order derivatives estimate for Monge-Amp\`ere equation of parabolic
type.  Chau-Chen-Yuan \cite{CCY}  used the  maximum principle by constructing a barrier function.
Later in \cite{DX}, considering  the drift Laplacian operator introduced by Colding-Minicozzi \cite {CM}, Ding-Xin used the integral method
and gave a complete improvement by dropping additional decay assumptions. In \cite{WW}, the third author reproved Ding-Xin's optimal result via a pointwise approach, which also works for a larger class of equations including Hessian quotient type.

If $\tau=\frac{\pi}{2}$, (\ref{e1.4}) becomes the special Lagrangian type equation
\begin{equation}\label{e1.6}
\sum_{i=1}^{n}\arctan\lambda_{i}(D^2 u)=-u+\frac{1}{2}\langle x, Du\rangle,\quad x\in\mathbb{R}^{n}.
\end{equation}
This equation was first introduced by Chau-Chen-He in \cite{CCH}.
Huang-Wang \cite{HW} proved that any smooth convex solution to (\ref{e1.6}) must be a quadratic polynomial.
Later, Chau-Chen-Yuan \cite{CCY} obtained the profound result that any smooth solution to
the special Lagrangian type equation must be a quadratic polynomial. Ding-Xin \cite{DX} also proved the same Bernstein type theorem  by the integral method.

In this paper, we settle the Bernstein problems  for the remaining cases of (\ref{e1.4}), i.e.,
 for $\tau=\frac{\pi}{4}$, $0<\tau<\frac{\pi}{4}$ and $\frac{\pi}{4}<\tau<\frac{\pi}{2}$, respectively.
Our main results are the following:
\begin{theorem}\label{t1.1}
Assume that $u$ is a $C^{2}$ solution of
\begin{equation}\label{e1.7}
 -\sqrt{2}\sum_{i=1}^n\frac{1}{1+\lambda_{i}}=-u+\frac{1}{2}\langle x, Du\rangle,\quad x\in\mathbb{R}^{n},
\end{equation}
that satisfies
\begin{equation*}
D^{2}u>-I \quad\ \text{or}\quad\  D^{2}u<-I.
\end{equation*}
Then $u$ must be a quadratic polynomial.  Furthermore, there exists a symmetric real matrix $A$ such that
\begin{equation*}
u=\sqrt{2}\sum_{i=1}^n\frac{1}{1+\bar{\lambda}_{i}}+\frac{1}{2}\langle x,Ax\rangle,
\end{equation*}
where $\bar{\lambda}_{i}$ $(1\leq i\leq n)$ are the eigenvalues of $A$.
\end{theorem}

\begin{theorem}\label{t1.2}
Assume that $u$ is a $C^{2}$ solution of
\begin{equation}\label{e1.8}
\frac{\sqrt{a^2+1}}{2b}\sum_{i=1}^n\ln\frac{\lambda_{i}+a-b}{\lambda_{i}+a+b}=-u+\frac{1}{2}\langle x, Du\rangle,\quad x\in\mathbb{R}^{n},
\end{equation}
that satisfies
$$D^2u>-(a-b)I \quad\ \text{or}\quad\  D^2u<-(a+b)I,$$ where $a=\cot \tau$, $b=\sqrt{\cot^2\tau-1}$, $0<\tau<\frac{\pi}{4}$.
Then $u$ must be a quadratic polynomial.  Furthermore, there exists a symmetric real matrix $A$ such that
\begin{equation*}
u=-\frac{\sqrt{a^2+1}}{2b}\sum_{i=1}^n\ln\frac{\bar{\lambda}_{i}+a-b}{\bar{\lambda}_{i}+a+b}+\frac{1}{2}\langle x,Ax\rangle,
\end{equation*}
where $\bar{\lambda}_{i}$ $(1\leq i\leq n)$ are the eigenvalues of $A$.
\end{theorem}

\begin{theorem}\label{t1.3}
Assume that $u$ is a $C^{2}$ solution of
\begin{equation}\label{e1.9}
\frac{\sqrt{a^2+1}}{b}\sum_{i=1}^n\arctan\frac{\lambda_{i}+a-b}{\lambda_{i}+a+b}=-u+\frac{1}{2}\langle x, Du\rangle,\quad x\in\mathbb{R}^{n},
\end{equation}
where $a=\cot \tau$, $b=\sqrt{1-\cot^2\tau}$, $\frac{\pi}{4}<\tau<\frac{\pi}{2}$.
Then $u$ must be a quadratic polynomial.  Furthermore, there exists a symmetric real matrix $A$ such that
\begin{equation*}
u=- \frac{\sqrt{a^2+1}}{b}\sum_{i=1}^n\arctan\frac{\bar{\lambda}_{i}+a-b}{\bar{\lambda}_{i}+a+b}+\frac{1}{2}\langle x,Ax\rangle,
\end{equation*}
where $\bar{\lambda}_{i}$ $(1\leq i\leq n)$ are the eigenvalues of $A$.
\end{theorem}

\begin{definition}\label{d1.10}
Let $\Gamma\subset\mathbb{R}^n$ be the open set that has the following properties:

(i) For any $\lambda\in\Gamma$,
\begin{equation*}
\frac{\partial F_\tau(\lambda)}{\partial\lambda_{i}}>0\quad\text{$i=1,...,n$}.
\end{equation*}

(ii) $\Gamma$ is symmetric, namely it is invariant under interchange of any two $\lambda_i$.

We say that $v$ is a $C^{2}$ admissible solution to (\ref{e1.1}), if
$\lambda(x,t)\in\Gamma$ for all $(x,t)$ in the domain of $v$; we say that $u$ is a $C^{2}$ admissible solution to (\ref{e1.4}), if
$\lambda(x)\in\Gamma$ for all $x$ in the domain of $u$.
\end{definition}

\begin{rem}
Equation \eqref{e1.4} is elliptic for admissible solutions. Our requirements for the Hessian of $u$ in Theorems \ref{t1.1} and \ref{t1.2}  are just admissiblity conditions.
Any admissible $C^{2}$ solutions to (\ref{e1.4}) must be smooth by the standard regularity theory for elliptic partial differential equations \cite{GT}.
\end{rem}
Moreover, our rigidity results also hold  for solutions on domains of $\mathbb{R}^n$ that blow up on the boundaries. These can be seen as analogues to the results of Trudinger-Wang (Theorem 2.1 in \cite{TW}),  Ding-Xin (Theorem 2.3 in \cite{DX}) and Li-Xu-Yuan (Theorem 1.1 in \cite{LXY}).
\begin{theorem}\label{t1.4}
Let $u$ be a $C^{2}$ admissible solution to \eqref{e1.7}, \eqref{e1.8} or \eqref{e1.9} on a domain $\Omega\subset\mathbb{R}^n$. Assume that $|Du|=\infty$ or $|u|=\infty$ on $\partial\Omega$. Then $\Omega=\mathbb{R}^n$ and $u$ is a quadratic polynomial.
\end{theorem}

Till now, the angle $\tau$ we consider for $g_\tau$ lies in $[0,\frac{\pi}{2}]$. How about the rigidity issue for $\tau $ in a broader range? It is obvious that $g_{\tau}=g_{\tau+2\pi}$ and $g_{\pi-\tau}|_{u}=g_{\tau}|_{-u}$, where $g_{\pi-\tau} |_u$ denotes the induced metric on the gradient graph of $u$ from $g_{\pi-\tau}$. When $\tau\in [-\frac{\pi}{2},-\frac{\pi}{4}]$, $g_\tau$ is negative semi-definite, so $g_{\tau}|_{u}$ cannot be a Riemannian metric. Thus the remaining case is $\tau\in (-\frac{\pi}{4},0)$. In this case, $g_{\tau}|_{u}$ is a Riemannian metric if and only if $$-(b+a)I<D^2u<(b-a)I,$$ where
$a=\cot \tau$, $b=\sqrt{\cot^2\tau-1}$. And the self-shrinker equation is
\begin{equation}\label{e1.10}
 \frac{\sqrt{a^2+1}}{2b}\sum_{i=1}^n\ln\frac{b+a+\lambda_{i}}{b-a-\lambda_{i}}=-u+\frac{1}{2}\langle x, Du\rangle.
\end{equation}
However, the rigidity phenomenon becomes very different in this case.
\begin{theorem}\label{t1.5}
Equation \eqref{e1.10} admits entire smooth non-quadratic admissible solutions.
\end{theorem}

The elliptic coefficients of the solution constructed in the proof of Theorem \ref{t1.5} decay fast (exponentially) in one direction. But under certain non-degeneracy conditions (in integral forms) on the elliptic coefficients, we can still draw the quadratic conclusion. For example, if the induced metric (see \eqref{induced metric} below) is complete, by Chen-Qiu's result (see Theorem 2 in \cite{CQ}), $u$ is quadratic. Here we give another non-degeneracy condition in terms of the image of gradient map.

\begin{theorem}\label{t1.6}
Let $u$ be a $C^{2}$ admissible solution to \eqref{e1.10} on a domain $\Omega\subset\mathbb{R}^n$. Assume the gradient map $(b+a)x+Du$ or $(b-a)x-Du$ from $\Omega$ to $\mathbb{R}^n$ is surjective. Then $\Omega=\mathbb{R}^n$ and $u$ is a quadratic polynomial.
\end{theorem}

\begin{rem}
There is a similar rigidity phenomenon for the equation
\begin{equation}\label{MSS}
\left(\delta_{ij}+\frac{f_if_j}{1-|Df|^2}\right)f_{ij}=-\frac{1}{2}f+\frac{1}{2}\langle x, Df\rangle,\quad |Df|<1.
\end{equation}
Suppose $f$ is an entire solution to \eqref{MSS} on $\mathbb{R}^n$. Then $M_f=\{(x,f(x))\,|\,x\in\mathbb{R}^n\}$ is an entire self-shrinker in the Minkowski space $\mathbb{R}^{n+1}_1$. If the induced metric is complete, by Chen-Qiu's result  \cite{CQ}, $f$ must be linear, corresponding to a hyperplane. On the other hand, without completeness assumption, we can construct entire non-trivial solutions in a similar way as in the proof of Theorem \ref{t1.5}.
This contrasts sharply with the Bernstein theorems for maximal hypersurfaces in the Minkowski spaces \cite{CY} and graphical self-shrinking hypersurfaces in the Euclidean spaces \cite{WL,LXY}.
\end{rem}

The organization of this paper is as follows. In the next section,  we deduce equation (\ref{e1.4}) and show its equivalency to the self-shrinker equation \eqref{e1.3} up to an additive constant. In Section 3, we prove Theorem \ref{t1.1} by the argument of an integral estimate with a suitable choice of a test function; in fact, as a byproduct of our argument, we give a Liouville type result for some $p$-Laplace type equations. In Sections 4 and 5,  we prove Theorem \ref{t1.2} and Theorem \ref{t1.3} respectively, via a pointwise approach. The proof of Theorem \ref{t1.4} is given in Section 6.
In the last section,  Theorem \ref{t1.5} is proved by
constructing a non-trivial solution to a second order ODE; and finally, Theorem \ref{t1.6} is proved by making use of  Legendre transform.

\section{Lagrangian self-shrinkers
in $(\mathbb{R}^n\times\mathbb{R}^n, g_\tau)$}
Throughout the following, Einstein's convention of summation over repeated indices is adopted. We denote, for a smooth function $u$,
 $$u_{i}=\dfrac{\partial u}{\partial x_{i}},\ u_{ij}=\dfrac{\partial^{2}u}{\partial x_{i}\partial x_{j}},\ u_{ijk}=\dfrac{\partial^{3}u}{\partial x_{i}\partial x_{j}\partial
x_{k}},\, \cdots.$$
\begin{Proposition}\label{prop2.1}
If (\ref{e1.1}) admits a smooth solution $v(x,t)$ on $\mathbb{R}^{n}\times\mathcal{I}$, where $\mathcal{I}$ is a time interval,
then (\ref{e1.2}) admits a smooth solution $\tilde X(x,t)$ on $\mathbb{R}^{n}\times\mathcal{I}$ with potential $v$.
In particular, there exists a family of diffeomorphisms
\begin{equation*}
{\psi}_{t}: \mathbb{R}^{n}\times\mathcal{I}\rightarrow\mathbb{R}^{n},
\end{equation*}
such that
\begin{equation*}
\tilde X_{t}=(\psi_{t}, Dv(\psi_{t},t))
\end{equation*}
is a solution to the Lagrangian mean curvature flow in $(\mathbb{R}^n\times\mathbb{R}^n, g_\tau)$.
\end{Proposition}

\begin{proof}
Let $e_i=(0,\cdots,1,\cdots, 0)$ be the $i$-th axis vector in $\mathbb{R}^{n}\times\mathbb{R}^{n}$, $i=1,2,\cdots, 2n$.
 And let $\langle \cdot,\cdot\rangle_\tau$ denote the inner product of $(\mathbb{R}^{n}\times\mathbb{R}^{n}, g_\tau)$. The tangential vector fields of $M_t=\{(x,Dv(x,t))\,|\,x\in\mathbb{R}^n\}$ are spanned by
$$E_i=e_i+v_{ij}e_{n+j},\ \ i=1,\cdots,n.$$
The induced metric on $M_t$ is given by
\begin{equation}\label{induced metric}
\begin{split}
g_{ij}&=\langle E_i,E_j\rangle_\tau\\
&=\langle e_i+v_{ik}e_{n+k},e_j+v_{jl}e_{n+l}\rangle_\tau\\
&=\sin\tau(\delta_{ij}+v_{ik}v_{kj})+2\cos\tau v_{ij}.
\end{split}
\end{equation}
Denote $(g_{ij})^{-1}$ by $(g^{ij})$. It is not hard to see that
$$g^{ij}=\frac{\partial F_\tau(\lambda(D^2 v))}{\partial v_{ij}}.$$
Let $\overline{\nabla}$ denote the Levi-Civita connection of $(\mathbb{R}^{n}\times\mathbb{R}^{n}, g_\tau)$. We have
$$\overline{\nabla}_{E_j}{E_i}=v_{ijk}e_{n+k}.$$
The mean curvature of $M_t$ is computed as
\begin{equation}\label{meancur}
\begin{split}
H&=g^{ij}\left(\overline{\nabla}_{E_i}{E_j}\right)^\perp\\
&=\left(g^{ij}v_{ijk}e_{n+k}\right)^\perp\\
&=\left(\partial_k F_\tau(\lambda(D^2 v))e_{n+k}\right)^\perp.
\end{split}
\end{equation}
By the evolution equation \eqref{e1.1} of $v$,
\begin{equation}\label{effectivemcf1}
H=\left(v_{tk}e_{n+k}\right)^\perp=(0,Dv_t)^\perp.
\end{equation}

We take a family of diffeomorphisms $$\psi_t:  \mathbb{R}^{n}\times\mathcal{I}\rightarrow\mathbb{R}^{n}$$ that satisfies
$$(I,D^2v)\cdot\frac{d\psi_{t}}{dt}=\left(0,Dv_t(\psi_t,t)\right)^\top,$$
where $\cdot$ denotes matrix product and $^\top$ denotes the projection to the tangent bundle of $M_t$.

Set $\tilde X_t=X_t\circ\psi_t$. It follows that
\begin{equation*}
\begin{split}
\frac{d\tilde X_t}{dt}&=\left(\frac{d\psi_{t}}{dt}, Dv_t(\psi_{t},t)+D^2v\cdot\frac{d\psi_{t}}{dt}\right)\\
&=\left(0,Dv_t(\psi_t,t)\right)^\perp\\
&=H(\tilde X_t).
\end{split}
\end{equation*}
This finishes the proof.
\end{proof}

\begin{definition}\label{d1.1}
We use (\ref{e1.1}) to denote special Lagrangian evolution equation in  $(\mathbb{R}^n\times\mathbb{R}^n, g_\tau)$.
\end{definition}

\begin{definition}\label{d1.2}
If $v(x,t)$ is a solution of (\ref{e1.1}) that satisfies
\begin{equation}\label{e2.1}
v(x,t)=-tv\left(\frac{x}{\sqrt{-t}},-1\right)
\end{equation}
for $t\in(-\infty,0)$, then we say that $v(x,t)$ is a self-shrinking solution to  special Lagrangian evolution equation (\ref{e1.1})
in  $(\mathbb{R}^n\times\mathbb{R}^n, g_\tau)$.
\end{definition}

\begin{Proposition}\label{prop2.3}
If $v(x,t)$ is a self-shrinking solution to  special Lagrangian evolution equation (\ref{e1.1}) in $(\mathbb{R}^n\times\mathbb{R}^n, g_\tau)$,
then $v(x,-1)$ satisfies equation (\ref{e1.4}).
\end{Proposition}
\begin{proof}
By (\ref{e1.1}) and (\ref{e2.1}), one can easily check that
\begin{equation*}
\begin{aligned}
F_\tau(\lambda(D^2 v(x,t)))&=\frac{\partial v}{\partial t}(x,t)\\
&=-v\left(\frac{x}{\sqrt{-t}},-1\right)+\frac{1}{2}\langle Dv\left(\frac{x}{\sqrt{-t}},-1\right),\frac{x}{\sqrt{-t}}\rangle.
\end{aligned}
\end{equation*}
By taking $t=-1$, we see that $v(x,-1)$ satisfies (\ref{e1.4}).
\end{proof}
\begin{Proposition}\label{prop2.4}
Let $u(x)$ be a smooth solution of equation (\ref{e1.4}). Define $$v(x,t)=-tu\left(\frac{x}{\sqrt{-t}}\right)\quad\,  t\in(-\infty,0).$$
Then $v(x,t)$ is a self-shrinking solution to  special Lagrangian evolution equation in $(\mathbb{R}^n\times\mathbb{R}^n, g_\tau)$.
\end{Proposition}
\begin{proof}
By a simple computation and equation (\ref{e1.4}), one checks directly that
\begin{equation*}
\begin{aligned}
\frac{\partial v}{\partial t}(x,t)&=-u\left(\frac{x}{\sqrt{-t}}\right)+\frac{1}{2}\langle Du\left(\frac{x}{\sqrt{-t}}\right),\frac{x}{\sqrt{-t}}\rangle.\\
&=F_\tau\left(\lambda\left(D^2 u\left(\frac{x}{\sqrt{-t}}\right)\right)\right)\\
&=F_\tau\left(\lambda\left(D^2v\left(x,t\right)\right)\right).
\end{aligned}
\end{equation*}
Thus we obtain the desired result.
\end{proof}

\begin{Proposition}\label{prop2.2}
If $v(x,t)$ is a self-shrinking solution to  special Lagrangian evolution equation in  $(\mathbb{R}^n\times\mathbb{R}^n, g_\tau)$,
then $M_{-1}\triangleq\{(x,Dv(x,-1))\,|\,x\in\mathbb{R}^n\}$ is a Lagrangian self-shrinker in $(\mathbb{R}^n\times\mathbb{R}^n, g_\tau)$.
\end{Proposition}

\begin{proof}
From (\ref{e2.1}) we see that
\begin{equation*}
\begin{aligned}
X_{t}(x)&=(x,Dv(x,t))\\
&=\sqrt{-t}\left(\frac{x}{\sqrt{-t}},Dv\left(\frac{x}{\sqrt{-t}},-1\right)\right)\\
&=\sqrt{-t}X_{-1}\left(\frac{x}{\sqrt{-t}}\right).\\
\end{aligned}
\end{equation*}
Then we have
\begin{equation}\label{ssevo}
\left(\frac{dX_{t}}{dt}\right)^{\perp}(x)=-\frac{1}{2}\frac{1}{\sqrt{-t}}X^\perp_{-1}\left(\frac{x}{\sqrt{-t}}\right).
\end{equation}
Equation \eqref{effectivemcf1} reads
\begin{equation}\label{effectivemcf2}
\left(\frac{dX_{t}}{dt}\right)^\perp=H.
\end{equation}
Combing \eqref{ssevo} and \eqref{effectivemcf2}, then taking $t=-1$, we obtain
\begin{equation*}
H=-\frac{1}{2}X^{\perp}_{-1}.
\end{equation*}
\end{proof}

\begin{Proposition}\label{prop2.4}
Let $M=\{(x,Du(x))\,|\,x\in\mathbb{R}^n\}$ be a Lagrangian self-shrinker in $(\mathbb{R}^n\times\mathbb{R}^n, g_\tau)$. Then $u$ satisfies equation \eqref{e1.4} up to an additive constant.
\end{Proposition}
\begin{proof}
By \eqref{meancur}, we have
\begin{equation}\label{meancur2}
H=\left(\partial_k F_\tau e_{n+k}\right)^\perp=(0,DF_\tau)^{\perp}.
\end{equation}
By definition, $H=-\frac{1}{2}{X^\perp}$. So we have
\begin{equation*}
(0,DF_\tau)^{\perp}=-\frac{1}{2}(x,Du)^{\perp}.
\end{equation*}
Namely for any $x\in\mathbb{R}^n$, $(x,D(2F_\tau+u))$ is spanned by $E_i$ $(1\leq i\leq n)$. Comparing the coefficients, we find that
\begin{equation*}
x_iu_{ij}=2\partial_jF_\tau+u_j\quad\text{for $j=1,...,n$.}
\end{equation*}
Namely
\begin{equation}\label{dss}
\partial_jF_\tau=\partial_j\left(-u+\frac{1}{2}\langle x, Du\rangle\right)\quad\text{for $j=1,...,n$.}
\end{equation}
Since $\mathbb{R}^n$ is connected, equation \eqref{dss} implies \eqref{e1.4} up to an additive constant.
\end{proof}

\section{Proof of Theorem \ref{t1.1}}

In this section, we give the proof of Theorem \ref{t1.1}. Two key ingredients are
Legendre transformation and an integral estimate. The cases $D^2u>-I$ and  $D^2u<-I$ are related by the transform
\begin{equation}\label{symtransform1}
u^-=-|x|^2-u,
\end{equation}
which also preserves \eqref{e1.7}. So we only need to consider the case $D^2u>-I$.

Let  $w=\frac{1}{2}|x|^2+u$ for $u$ as in Theorem \ref{t1.1}.
Then straight calculations yield $\lambda(D^2w)=\lambda(I+D^2u)$ and
$$-u+\frac{1}{2}\langle x,Du\rangle=-w+\frac{1}{2}\langle x,Dw\rangle.$$ Thus $w$ is strictly convex and satisfies
\begin{equation}\label{effective1}
-\sqrt{2}\sum^n_{i=1}\frac{1}{\mu_i}=-w+\frac{1}{2}\langle x,Dw\rangle,
\end{equation}
where $\mu=(\mu_1,...,\mu_n)$ are the eigenvalues of $D^2w$.
Take the Legendre transform
$$y=Dw,\qquad  w^*(y)=\langle x,Dw\rangle -w. $$

\noindent Then $D^2w^*=(D^2w)^{-1}$ and
$$\frac{1}{2}\langle x,Dw^*\rangle-w^*=-w+\frac{1}{2}\langle x,Dw\rangle.$$

\noindent Therefore $w^*$ is also strictly convex. By (\ref{e1.7}), $w^*$ satisfies the following equation
\begin{equation}\label{eq3.1}
\sqrt{2}\Delta w^*=\frac{1}{2}\langle y,Dw^*\rangle -w^*.
\end{equation}

Next we show that
\begin{claim}
$Dw(\mathbb{R}^n)=\mathbb{R}^n$, namely the domain of $w^*$ is the entire $\mathbb{R}^n$.
\end{claim}

\begin{proof}
Since $w$ is strictly convex, for any fixed $\theta\in\mathbb{S}^{n-1}_1$, $Dw(r\theta)\cdot\theta$ monotonically increases in $[0,+\infty)$, where ``$\cdot$'' denotes the standard Euclidean inner product. Suppose that there exist some $\theta_0\in\mathbb{S}_1^{n-1}$ and $\beta_0>0$ such that $$\lim_{r\rightarrow+\infty}Dw(r\theta_0)\cdot\theta_0\leq\beta_0<\infty.$$
Due to the strict convexity of $w$, we have
\begin{align*}
\frac{1}{2}r\theta_0\cdot Dw(r\theta_0)-w(r\theta_0)&\geq\frac{1}{2}r\theta_0\cdot Dw(r\theta_0)-[r\theta_0\cdot Dw(r\theta_0)+w(0)]\\
&=-\frac{1}{2}r\theta_0\cdot Dw(r\theta_0)-w(0)\\
&\geq -\frac{\beta_0}{2}r-w(0).
\end{align*}
By \eqref{effective1}, $w(0)>0$ and
\begin{equation*}
\sqrt{2}\sum^n_{i=1}\frac{1}{\mu_i(r\theta_0)}\leq \frac{\beta_0}{2}r+w(0).
\end{equation*}
It follows that
\begin{equation*}
\mu_{\min}(r\theta_0)>\frac{1}{\beta_0r+2w(0)}.
\end{equation*}
Then we have
\begin{equation}\label{contrad1}
\begin{split}
Dw(r\theta_0)\cdot\theta_0-Dw(0)\cdot\theta_0&=\int^{r}_{0}\theta^T_0\cdot D^2w(t\theta_0)\cdot\theta_0\,dt\\
&\geq \int^{r}_{0}\frac{dt}{\beta_0t+2w(0)}=\frac{1}{\beta_0}\ln\left(1+\frac{\beta_0}{2w(0)}r\right).
\end{split}
\end{equation}
As $r\rightarrow+\infty$, the left-hand side of \eqref{contrad1} is finite by the assumption, while the right-hand side of \eqref{contrad1} blows up. Thus we get a contradiction. Hence, for any $\theta\in\mathbb{S}_1^{n-1}$, $$\lim_{r\rightarrow+\infty}Dw(r\theta)\cdot\theta=+\infty.$$ This implies $w$ has a superlinear growth. It follows that
for any $y\in\mathbb{R}^n$, $y\cdot x-w(x)$ attains its global maximum of $\mathbb{R}^n$ at some finite point $x_0$. At $x_0$, we must have $Du(x_0)=y$.
Since $y$ is arbitrary, we conclude that $Dw(\mathbb{R}^n)=\mathbb{R}^n$.
\end{proof}

Up to this moment, to prove Theorem \ref{t1.1}, we need only to show
that any strictly convex entire solution of (\ref{eq3.1}) is a quadratic polynomial.

Denote
\begin{equation}\label{phase1}
\phi:=\frac{1}{2}\langle y,Dw^*\rangle -w^*=\frac{1}{2}y_l{w^*}_l-w^*.
\end{equation}

\noindent Recall that the repeated indices are summed. Taking derivatives of \eqref{phase1} twice, we obtain
\begin{equation}\label{eq3.2}
\phi_{ij}=\frac{1}{2}y_l{w^*}_{ijl}.
\end{equation}
A differentiation of \eqref{eq3.1} with respect to $x_l$ yields
\begin{equation}\label{diffu1}
\sqrt{2}\Delta w^*_{l}=\phi_{l}.
\end{equation}
Combining \eqref{eq3.2} and \eqref{diffu1}, we get
\begin{equation}\label{eq3.3}
\Delta\phi=\frac{\sqrt{2}}{4}y_i\phi_i.
\end{equation}

We will show that any positive entire solutions of (\ref{eq3.3}) must be constant.
Then by  (\ref{eq3.2}), the strictly convex entire solution $w^*$ of  (\ref{eq3.1}) must be a
quadratic polynomial, and Theorem \ref{t1.1} is proved. In fact, we will prove the following
more general result:

\begin{Proposition}\label{prop3.1}
Let $p> 1$, $K>0$ be constants and $h\geq 0$ satisfies
\begin{equation}\label{eq3.4}
\diver\left(|Dh|^{p-2}Dh\right)=Kx_ih_i|Dh|^{p-2} \qquad \text{on} \ \  \mathbb{R}^n.
\end{equation}
Then $h$ must be constant.

\end{Proposition}

{\bf Proof}\quad  Let $\eta$ be a smooth cut-off function satisfying:
\begin{equation}\label{eq3.6}
\begin{cases}
                          \eta\equiv 1  &\text{in} \,\,B_R,\\
                        0\leq\eta\leq1  &\text{in} \,\,B_{2R},\\
                          \eta\equiv 0  &\text{in} \,\,\,\mathbb{R}^n\backslash B_{2R},\\
     |D \eta|\lesssim \frac{1}{R}  &\text{in} \,\,\,\mathbb{R}^n,
\end{cases}
\end{equation}

\noindent where and in the sequel, $B_R$ denotes a
ball in $\mathbb{R}^n$ centered at the origin with radius $R$;
and we use  ``$\lesssim $" ,  ``$\backsimeq$", etc. to drop out some
positive constants independent of $R$ and $h$.

  By divergence theorem, with $\rho=-\frac{K}{2}|x|^2-h$, we compute

\begin{equation}\label{eq3.7}
\begin{split}
 \int_{\mathbb{R}^n} D_i\left(|Dh|^{p-2}h_i\right)e^{\rho}\eta ^{p}dx
 = & -\int_{\mathbb{R}^n} |Dh|^{p-2}h_iD_i\left(e^{\rho}\eta^{p}\right)dx\\
 = & -\int_{\mathbb{R}^n} |Dh|^{p-2}h_i\left({\rho}_i\eta ^{p}+p\eta_i\eta^{p-1}\right)e^{\rho}dx\\
 = & \int_{\mathbb{R}^n} Kx_ih_i|D h|^{p-2}e^{\rho}\eta^{p}dx+\int_{\mathbb{R}^n} |D h|^{p}e^{\rho}\eta^{p}dx\\
 \,& -p\int_{\mathbb{R}^n} |Dh|^{p-2}h_i\eta_i\eta ^{p-1}e^{\rho}dx.
\end{split}
\end{equation}

\noindent Using equation (\ref{eq3.4}) in  (\ref{eq3.7}) we get

\begin{equation}\label{eq3.8}
\begin{split}
 \int_{\mathbb{R}^n} |Dh|^{p}e^{\rho}\eta ^{p}dx
 &=p\int_{\mathbb{R}^n} |Dh|^{p-2}h_i\eta_i\eta ^{p-1}e^{\rho}dx\\
 &\lesssim \frac{1}{R}\int_{\mathbb{R}^n} |Dh|^{p-1}\eta^{p-1}e^{\rho}dx\\
 &\lesssim \varepsilon\int_{\mathbb{R}^n} |Dh|^{p}e^{\rho}\eta ^{p}dx
          +\frac{1}{R^p}\int_{B_{2R}} e^{\rho}dx,
\end{split}
\end{equation}
\noindent where in the last step, we have used the Young's inequality with exponent pair $(\frac{p}{p-1},p)$.
For $h\geq 0$ we have
$$\int_{\mathbb{R}^n} e^{\rho}dx<+\infty.$$

Taking $\varepsilon$ small and letting $R\rightarrow+\infty$ in (\ref{eq3.8}), one must get $|Dh|\equiv 0$ in $\mathbb{R}^n$. The proof of the Proposition \ref{prop3.1} is completed. $\qed$

\section{Proof of Theorem \ref{t1.2}}
We prove Theorem \ref{t1.2} by constructing a barrier function. The case $D^2u>-(a-b)I$ and the case $D^2u<-(b+a)I$ are related by the transform
$$u^-=-a|x|^2-u,$$
which also preserves \eqref{e1.8}. So we only need to consider the case $D^2u>-(a-b)I$.
\begin{proof}
Take $w=u+\frac{a-b}{2}|x|^2$. Then $w$ is strictly convex and satisfies
\begin{equation}\label{Meq'}
\frac{\sqrt{a^2+1}}{2b}\sum^n_{i=1}\ln\left(\frac{\mu_i}{\mu_i+2b}\right)=-w+\frac{1}{2}\langle x,Dw\rangle,
\end{equation}
where $\mu=(\mu_1,...,\mu_n)$ are the eigenvalues of $D^2w$.
Define
\begin{equation*}
F\left(D^2w\right)=\frac{\sqrt{a^2+1}}{2b}\sum^n_{i=1}\ln\left(\frac{\mu_i}{\mu_i+2b}\right),
\end{equation*}
and the coefficients
\begin{equation*}
a^{ij}\left(D^2w\right)=\frac{\partial F\left(D^2w\right)}{\partial w_{ij}}.
\end{equation*}
At any point, we can rotate the coordinates to diagonalize $D^2w$ and $(a^{ij})$ simultaneously. In the new coordinates, $D^2w$ and $(a^{ij})$ are $\diag\{\mu_1,...,\mu_n\}$ and $\diag\left\{\frac{\sqrt{a^2+1}}{\mu_1(\mu_1+2b)},...,\frac{\sqrt{a^2+1}}{\mu_n(\mu_n+2b)}\right\}$ respectively. As $w$ is convex, $\left(a^{ij}\right)$ is positive-definite. And we have
\begin{equation*}
a^{ij}w_{ij}=\sum^n_{k=1}\frac{\sqrt{a^2+1}}{\mu_k+2b}<\frac{n\sqrt{a^2+1}}{2b}.
\end{equation*}
Define the phase
\begin{equation}\label{phase}
\phi(x)=-w+\frac{1}{2}\langle x,Dw\rangle.
\end{equation}
By \eqref{Meq'}, $\phi<0$. Taking derivatives of \eqref{phase} twice, we obtain
\begin{equation}\label{diffphase}
\phi_{ij}=\frac{1}{2}x_sw_{ijs}.
\end{equation}
A differentiation of \eqref{Meq'} with respect to $x_s$ yields
\begin{equation}\label{diffu}
a^{ij}w_{ijs}=\phi_{s}.
\end{equation}
Combining \eqref{diffphase} and \eqref{diffu}, we get
\begin{equation*}
a^{ij}\phi_{ij}-\frac{1}{2}x\cdot D\phi=0.
\end{equation*}
Thus $\phi$ satisfies an elliptic equation without zeroth order term. The corresponding elliptic operator is
\begin{equation*}
\mathcal{L}=a^{ij}\partial^2_{ij}-\frac{1}{2}x\cdot D.
\end{equation*}
Define $\tilde w(x)=w(x)-Dw(0)\cdot x$. We have
\begin{equation*}
\mathcal{L}\tilde w=a^{ij}\tilde w_{ij}-\frac{1}{2}x\cdot D\tilde w<\frac{n\sqrt{a^2+1}}{2b}-\frac{1}{2}x\cdot D\tilde w.
\end{equation*}
Obviously, $\tilde w$ is strictly convex and $D\tilde w(0)=0$. So $\tilde w$ is proper. Thus there exists $R_0>0$ such that for $x\in\mathbb{R}^n\setminus B_{R_0}$, $\tilde w(x)\geq |\tilde w(0)|+\frac{n\sqrt{a^2+1}}{b}.$ Due to the convexity of $\tilde w$, $$x\cdot D\tilde w(x)\geq\tilde w(x)-\tilde w(0).$$ Hence, for $x\in\mathbb{R}^n\setminus B_{R_0}$, $\mathcal{L}\tilde w<0$.

For any $\varepsilon>0$, take a barrier function $b_{\varepsilon}(x)$ defined by
\begin{equation*}
b_{\varepsilon}(x)=\varepsilon\tilde w(x)+\max_{\partial B_{R_0}}\phi.
\end{equation*}
Clearly we have
\begin{equation*}
\mathcal{L}b_{\varepsilon}\leq\mathcal{L}\phi\quad\mbox{for $x\in\mathbb{R}^n\setminus B_{R_0}$},
\end{equation*}
and
\begin{equation*}
b_{\varepsilon}(x)\geq\phi(x)\quad\mbox{on $\partial B_{R_0}$}.
\end{equation*}
Since $\phi<0$ and $\tilde w$ is proper, we also have
\begin{equation*}
b_{\varepsilon}(x)>\phi(x)\quad\mbox{as $x\rightarrow+\infty$}.
\end{equation*}

The weak maximum principle then implies
\begin{equation*}
\varepsilon\tilde w(x)+\max\limits_{\partial B_{R_0}}\phi\geq\phi(x)\quad\mbox{for $x\in\mathbb{R}^{n}\setminus B_{R_0}$}.
\end{equation*}
Letting $\varepsilon\rightarrow 0$, we obtain
\begin{equation*}
\max\limits_{\partial B_{R_0}}\phi\geq\phi(x)\quad\mbox{for $x\in\mathbb{R}^{n}\setminus B_{R_0}$}.
\end{equation*}
So $\phi$ attains its global maximum in the closure of $B_{R_0}$. Hence $\phi$ is a constant by the strong maximum principle. Then by \eqref{diffphase}, $w$ must be a quadratic polynomial. So is $u$.
\end{proof}

\section{Proof of Theorem \ref{t1.3}}
We prove Theorem \ref{t1.3} by transforming \eqref{e1.9} into \eqref{e1.6}.
\begin{proof}
By the difference formula for tangent, we have
\begin{equation*}
\arctan\frac{\lambda_{i}+a-b}{\lambda_{i}+a+b}=\arctan\left(\frac{\lambda_i+a}{b}\right)-\frac{\pi}{4}.
\end{equation*}
Take
\begin{equation}\label{transform1}
w(x)=\frac{b}{\sqrt{a^2+1}u\left(\frac{(a^2+1)^{\frac{1}{4}}}{b}x\right)}+\frac{a}{2b}|x|^2-\frac{n\pi}{4}.
\end{equation}
Then $w$ satisfies
\begin{equation*}
\sum_{i=1}^n\arctan\mu_{i}=-w+\frac{1}{2}\langle x, Dw\rangle,
\end{equation*}
where $\mu=(\mu_1,...,\mu_n)$ are the eigenvalues of $D^2w$. According to Theorem 1.1 in \cite{CCY} or Theorem 1.2 in \cite{DX}, $w$ is a quadratic function. So is $u$. This finishes the proof.
\end{proof}

\section{Proof of Theorem \ref{t1.4}}
\begin{proof}
Suppose $u$ is an admissible solution to \eqref{e1.4}. Fix any $\theta\in\mathbb{S}_1^{n-1}$, consider a function $q_\theta(r)$ defined on $(0,+\infty)$ by
\begin{equation*}
q_\theta(r)=\frac{u(r\theta)}{r^2}.
\end{equation*}
By equation \eqref{e1.4},
\begin{equation*}
q_\theta'(r)=\frac{ru_r(r\theta)-2u(r\theta)}{r^3}=\frac{2F_\tau(\lambda(D^2u(r\theta)))}{r^3}.
\end{equation*}

Case 1. If $u$ is an admissible solution to \eqref{e1.7}, then $D^2u>-I$ or $D^2u<-I$. We first assume $D^2u>-I$. By \eqref{e1.7}, we have $$q_\theta'(r)<0.$$
So for $r\geq 1$,
$$\frac{u(r\theta)}{r^2}\leq u(\theta).$$
This implies that for $x\in\mathbb{R}^n\setminus B_1$, $u(x)\leq C|x|^2$, where $C=\max_{\partial B_1}u$. Since $D^2u>-I$, $u+|x|^2$ is convex. Hence $u$ has at most a quadratic growth. It follows that $Du$ has at most a linear growth. For the case $D^2u<-I$, we can get the same growth estimates via the transform \eqref{symtransform1}. Consequently, if $|Du|$ or $u$ blows up on $\partial\Omega$, $\Omega$ must be $\mathbb{R}^n$. The proof for equation \eqref{e1.8} is almost the same as above proof for equation \eqref{e1.7}.

Case 2. If $u$ is a solution to \eqref{e1.9}, we have
\begin{equation*}
|q_\theta'(r)|<\frac{n\pi}{2}r^{-3}.
\end{equation*}
Since $r^{-3}\in L^1([1,+\infty))$, $q_\theta'(r)$ is integrable. It follows that $u$ has at most a quadratic growth. Therefore, if $u$ blows up on $\partial\Omega$, then $\Omega$ must be the entire $\mathbb{R}^n$. If $|Du|$ blows up on $\partial\Omega$, by transform \eqref{transform1} and Theorem 1.1 in \cite{LXY}, we also conclude that $\Omega=\mathbb{R}^n$.
\end{proof}

\section{Proof of Theorems \ref{t1.5} and \ref{t1.6}}
{\bf Proof of Theorem \ref{t1.5}}\quad
For simplicity, we take the following trivial transform
$$w(x)=\frac{2b}{\sqrt{a^2+1}}u\left(\frac{(a^2+1)^{\frac{1}{4}}}{2b}x\right)+\frac{a+b}{4b}|x|^2.$$
Then $w$ satisfies
\begin{equation}\label{effective2}
\sum^n_{i=1}\ln\frac{\mu_i}{1-\mu_i}=-w+\frac{1}{2}\langle x,Dw\rangle.
\end{equation}
where $\mu=(\mu_1,...,\mu_n)$ are the eigenvalues of $D^2w$. And $$0<D^2w<I.$$
We first construct a nontrivial one dimensional solution to \eqref{effective2}, namely a solution to
\begin{equation}\label{ODE}
\frac{w_1''}{1-w_1''}=\exp\left(\frac{1}{2}tw_1'-w_1\right),
\end{equation}
where $w_1=w_1(t)$.
Then a simple combination
$$w(x)=w_1(x_1)+\frac{|x|^2-x_1^2}{4}$$
gives a nontrivial solution of arbitrary dimension.

Let $\phi=\frac{1}{2}tw_1'-w_1$. Similar calculations as in the proof of Theorem \ref{t1.1} give
\begin{equation}\label{phaseODE}
\phi''=\frac{e^\phi}{2(1+e^\phi)^2}t\phi'.
\end{equation}

For any $a_0$, $a_1\in\mathbb{R}$, by the Cauchy-Kovalevskaya Theorem, there exists a unique local solution around $t=0$ to \eqref{phaseODE} that satisfies $\phi(0)=a_0$, $\phi'(0)=a_1$. If $a_1=0$, then $\phi(x)\equiv a_0$, corresponding to a trivial solution. So without loss of generality, we assume $a_1>0$. Next we show this local solution can be extended to the whole $\mathbb{R}$.

Suppose $(-l_0,l_1)$ is the maximal interval where the solution exists. In the following, we show $l_{1}=+\infty$.
As $\phi'(0)=a_1>0$, from equation \eqref{phaseODE}, we see $\phi'$ monotonically increases in $[0,l_1)$. So $\phi(x)\geq a_1t+a_0$ in $[0,l_1)$. It follows that
$$\phi''(t)\leq te^{-a_1t-a_0}\phi'(t).$$
Namely
$$(\ln \phi')'\leq te^{-a_1t-a_0}.$$
Since $te^{-a_1t}\in L^1([0,+\infty))$, $\phi'(t)$ is bounded in $[0,l_1)$. Consequently, we conclude that $l_1=+\infty$.
By the similar method, we can also obtain $l_0=+\infty$.

From the relation $\phi=\frac{1}{2}tw_1'-w_1$, we obtain $w_1(0)=-a_0$ and $w_1'(0)=-2a_1$. Equation \eqref{ODE} reads $$w_1''=\frac{e^{\phi}}{1+e^{\phi}}.$$
Integrating above equation twice, we get an entire non-quadratic solution to \eqref{ODE}. It is easy to see that there is a corresponding non-quadratic  entire admissible solution to \eqref{e1.10}. $\qed$

Next we prove Theorem \ref{t1.6} by Legendre transform.

{\bf Proof of Theorem \ref{t1.6}}\quad
The assumption that the gradient map $(b+a)x+Du$ or $(b-a)x-Du$ is surjective corresponds to $Dw$ or $x-Dw$ is surjective.

If $Dw$ is surjective, take the Legendre transformation
$$y=Dw,\qquad  w^*(y)=\langle x,Dw\rangle -w. $$

\noindent Then $D^2w^*>I$ and $w^*$ satisfies
$$\sum^n_{i=1}\ln(\mu^*_i-1)=-w^*+\frac{1}{2}\langle y,Dw^*\rangle,$$
where $\mu^*=(\mu^*_1,...,\mu^*_n)$ are the eigenvalues of $D^2w^*$. Since $Dw$ is surjective, $w^*$ is defined on the whole $\mathbb{R}^n$.  Then $w^*-\frac{|x|^2}{2}$ is an entire convex solution on $\mathbb{R}^n$ to \eqref{e1.5} up to a trivial affine transform. According to Theorem 1.1 in \cite{DX}, $w^*-\frac{|x|^2}{2}$ must be a quadratic polynomial. So are $w^*$, $w$ and $u$. It follows that $\Omega=\mathbb{R}^n$.

If $x-Dw$ is surjective, let $w^-=\frac{|x|^2}{2}-w$. Then $w^-$ is convex and satisfies \eqref{effective2}. Also, $Dw^-$ is surjective. According to the discussion above, $w^-$ is a quadratic function. So are $w$ and $u$. It also follows that $\Omega=\mathbb{R}^n$. This completes the proof.$\qed$

\textbf{Aknowledgement}
Rongli Huang and Qianzhong Ou would   like to thank Professors Jiguang Bao, Xinan Ma and Yuanlong Xin for their interests and constant encouragement. Wenlong Wang is grateful to  Professors Yuguang Shi,    Yu Yuan and   Xiaohua Zhu  for their encouragement and constant support.

\end{document}